\documentclass[authoryear,preprint,review,12pt]{elsarticle}

\usepackage{amssymb}
\usepackage{amsthm}
\usepackage{bbm}
\usepackage{mathtools}
\usepackage{amsmath, color}
\theoremstyle{definition}
\newtheorem{theorem}{Theorem}[section]

\theoremstyle{definition}

\theoremstyle{definition}
\newtheorem{lemma}[theorem]{Lemma}
\newtheorem{corollary}[theorem]{Corollary}
\newtheorem{example}[theorem]{Example}
\newtheorem{remark}[theorem]{Remark}

\newcommand{\gcolor}[1]{\textcolor{blue}{#1}}

\journal{Statistics and Probability Letters}

\begin{document}

\begin{frontmatter}



\title{A Phase Transition Phenomenon for Ruin Probabilities in a Network of Agents and Objects}

\author[1]{Rukuang Huang}
\ead{rukuang.huang@jesus.ox.ac.uk}

%

\address[1]{{Big Data Institute, University of Oxford},
           {Li Ka Shing Centre for Health Information and Discovery, Old Road Campus},
           {Oxford},
            {OX3 7LF},
            {United Kingdom}}




\begin{abstract}
The classical Cramér-Lundberg risk process models the ruin probability of
an insurance company experiencing an incoming cash flow - the premium income, and an outgoing cash flow - the claims. 
From a system's viewpoint, the web of insurance agents and risk objects can be represented by a bipartite network. In such a bipartite network setting, it has been shown that joint ruin of a group of agents
may be avoided even if individual agents would experience ruin in the classical Cramér-Lundberg model. This paper describes and examines a phase transition phenomenon for these ruin probabilities.
\end{abstract}







\end{frontmatter}


\section{Introduction}
\label{Intro}
Systemic risk  in an insurance market reflects its vulnerability 
to events such as earthquakes, epidemics (COVID-19) and tsunami, through cascading losses in an inter-related system of agents and objects. This  system can be represented as
a bipartite network of interacting agents such as insurance companies,
or different business lines of an insurance company, 
and objects, with an edge between an agent and an object if the agent insures that object.

Ignoring network effects, if the system consists of only one agent and  one object, a classical model for ruin is the 
Cramér-Lundberg model, see for example \cite{asmussen2010ruin}, where also multivariate non-network based  
extensions can be found.
From a network viewpoint, as the strategies of agents are usually unknown, a model for a bipartite insurance market assumes that edges are modelled as random.
This paper  builds on the results of \cite{ruinprobgesine}, in which, under a bipartite random network and an exponential system, an expression for the summative ruin of a group $Q$ of agents is 
derived in terms of the so-called {\it network Pollaczek-Khintchine random variable} $P^Q$, see \eqref{ruinparameter} below; when $\mathbbm{P}(P^Q \ge 1) = 1$ then the summative ruin probability of the group $Q$ is 1; otherwise it is less than 1.
The set $Q$ can be thought of as a collection of agents in which resources can be transferred between agents, e.g. between different business line of a company. In this paper we show that this model undergoes a phase transition  - There is a threshold, which depends on the parameters of the model, above which $\mathbbm{P}(P^Q < 1) \rightarrow 1$, and below which  $\mathbbm{P}(P^Q < 1) \rightarrow 0$. 
In the regime $\mathbbm{P}(P^Q < 1) \rightarrow 1$, the model could be used to optimize the size of $Q$ for minimising the summative ruin probability.

This paper is structured as follows. Section \ref{setting} sets up the model. 
Section \ref{phase_transition} describes the phase transition phenomenon and provides a  theoretical explanation for it, as well as 
an asymptotic result in a moderately dense network when the number of objects $d$ tends to infinity while $Q$ is fixed. {Finally, t}he results are illustrated by a bipartite Bernoulli random graph.

\section{Model setting}
\label{setting}
For the ruin model, suppose we have $q$ agents and $d$ objects. For each object $j \in \{1,...,d\}$, $c_j$ is the constant premium rate, $\{N_j(t)\}_{t\geq 0}$ is a Poisson process with intensity $\lambda>0$ which models the arrivals of the claim, independent of all other random variables, and $\{X_j(k)\}_{k\in \mathbb{N}}$ are exponentially distributed and i.i.d. positive claim sizes with finite mean $\mu_j$. We define the \textit{insurance risk process} $V_j$ for object $j$ as 
$$
    V_j(t) = \sum_{k=1}^{N_j(t)} X_j(k) - c_j t, \quad t\geq 0 . 
$$
We assume that the claim arrival processes and claim size distributions are independent across $j$, so that the $V_j$'s are independent.

Next, we construct a random undirected weighted bipartite network, which is independent of the insurance risk processes for all $j$, between the agents and the objects. Here, agents insure objects but agents do not insure each other. Let the edge indicator $\mathbbm{1}\{i\sim j\}$ denote whether there is an edge between agent $i$ and object $j$, an edge representing that agent $i$ insured object $j$. 
Let $Q\subseteq\{1,\dots,q\}$ be nonempty and $j\in \{1,\dots,d\}$.
Let
$$\mathbbm{1}\{Q\sim j\} := \max_{i\in Q} \mathbbm{1}\{i \sim j\}
.
$$
Then the weighted adjacency matrix $A$ of the bipartite network is $  A = (A^i_j)_{\substack{i=1,\dots,q\\j=1,\dots,d}}$, where 
$$
    A^i_j = \mathbbm{1}\{i\sim j\} W^i_j, 
\mbox{ with } 
    W^i_j = \frac{\mathbbm{1}\{Q\sim j\}r^Q}{\sum_{k\in Q} \mathbbm{1}\{k\sim j\} \mu_j}, \quad \frac{0}{0} := 0.
$$
Here $r^Q$ is a constant depending only on $Q$, 
such that 
\begin{equation}
\label{weightcondition}
    0\leq \sum_{i=1}^q A^i_j = \sum_{i=1}^q \mathbbm{1}\{i\sim j\}W^i_j \leq 1 \textrm{ for each } j.
\end{equation}
It is easy to prove that 
\begin{equation}
\label{rQ}
    r^Q = \frac{\min_j \mu_j}{q-|Q|+1}
\end{equation}
satisfies \eqref{weightcondition}. 
Although agents in $Q$ connected to object $j$ share an equal amount of the loss, the amount is inversely proportional to the expected claim size of object $j$ and there could potentially be uninsured loss. In \cite{ruinprobgesine}, these weights are thus called \textit{proportional weights}.

Each object $j$ is assigned its own insurance risk process $V_j$ independently and similarly each agent $i$ is assigned its portfolio $R^{(i)}$, given by
$$
    R^{(i)}(t) := \sum_{j=1}^d A^i_j V_j(t), \text{for } t\geq 0.
$$
Thus, for
each agent $i$, $R^{(i)}$ is a weighted sum of the portfolio losses of the objects which are connected to agent $i$, with $W^i_j$ as the weights. If $i$ is connected to $j$, then $W^i_j$ represents the proportion of loss due to object $j$ suffered by agent $i$. 
By construction, the components of $R(t) = (R^{(1)}(t),\dots,R^{(q)}(t))$ are not independent. 

The main object of interest is the ruin probability for the sum of the ruins 
in $Q$, defined as (\cite{ruinprobgesine})
$$
    \Psi^Q(u) := \mathbb{P}\left(\sum_{i\in Q} (R^{(i)}(t)- u^{(i)}) \geq 0 \text{ for some } t\geq 0\right)
$$
where $0 \le u^{(i)} < \infty, i=1, \ldots, q$ are the risk reserves for the agents.
In \cite{ruinprobgesine} the following result is shown.
\begin{theorem}[Theorem 4.1 in \cite{ruinprobgesine}]
Let $\sum_{i\in Q} u^{(i)} > 0.$ 
Then
\begin{equation}
\label{ruinexpression}
    \Psi^Q(u)=\mathbb{P}(P^Q <1)\mathbb{E}\left[P^Q e^{-\frac{1-P^Q}{r^Q}\sum_{i\in Q}u^{(i)}}|P^Q<1\right] +\mathbb{P}(P^Q\ge 1),
\end{equation}
where $P^Q$, the network Pollaczek–Khintchine random variable, is given by
\begin{equation}
\label{ruinparameter}
    P^Q=\lambda \frac{\sum_{j=1}^d \mathbb{I}\{Q\sim j\}}{\sum_{j=1}^d \mathbb{I}\{Q\sim j\} c_j/\mu_j}, \quad\frac{0}{0}:=0.
\end{equation}
\end{theorem}

In this paper, we assume that the network between agents and objects is generated by a 
{\it stochastic blockmodel}, which is a standard model for networks, see for example Chapter 12 in \cite{newman2018networks}. In this model, there are $K$ types of agents and $L$ types of objects. Let $s(i)$ be the type of agent $i$ and assume $\mathbb{P}(s(i)=k) = w_k, \forall k \in \{1,\dots,K\}$, independently with $\sum_{k=1}^K w_k=1$. Similarly,  for each object $j$, independently, $\mathbb{P}(t(j)=l) = v_l, \forall l\in \{1,\dots,L\}$, with $\sum_{l=1}^Lv_l=1$. Further, we let  $\mathbb{P}(i\sim j|s(i)=k,t(j)=l) = p_{kl}$ for each $k$ and $l$; edge probabilities between two objects or between two agents are set to be zero.   The edge indicators are assumed to be mutually independent given the type configuration. Hence by conditioning on the type configuration, for each $i$ and $j$, $\mathbb{P}(i\sim j) = \sum_{k=1}^K\sum_{l=1}^L p_{kl}w_k v_l$. 
In particular, let $C=\{s(1),\ldots,s(q),t(1),\ldots,t(d)\}$ be the random type configuration of agents and objects, and let $c$ be a realisation of $C$. Then  
\begin{align*}
\mathbb{P}(Q\sim j) 
&= \sum_c \left(1-\prod_{i\in Q}(1-\mathbb{P}(i\sim j|C=c))\right)\mathbb{P}(C=c)\\
&= \sum_{k_1=1}^K \dots \sum_{k_{|Q|=1}}^K \sum_{l=1}^L \left(1-\prod_{i\in Q} (1-p_{k_i l})\right)v_l w_{k_1}...w_{k_{|Q|}}.
\end{align*}

\section{A phase transition phenomenon}
\label{phase_transition}
The ruin probability $\Psi^Q(u)$ is crucial for assessing whether or not the group $Q$ of agents jointly experiences ruin. Often it is impossible to compute \eqref{ruinexpression} exactly, so we employ the following Monte-Carlo estimator for $\Psi^Q(u)$:
\begin{equation}
    \widehat{\Psi^Q(u)}=\frac{1}{B}\sum_{b=1}^B \left[\mathbb{I}\{P^Q_b<1\}P^Q_b e^{-\frac{1-P^Q_b}{r^Q}\sum_{i\in Q}u^i}+\mathbb{I}\{P^Q_b>1\}\right],
\end{equation}
where $P^Q_b$ is a sample of $P^Q$ obtained by simulating the random network and computing \eqref{ruinparameter}, and $B$ is the number of Monte-Carlo samples.

Here we simulate the underlying network as a bipartite Bernoulli network with edge probability $0.5$. We take $q=d=10, u^i=1,\mu_j=1,\lambda=1$, 
and $r^Q$ is 
as in \eqref{rQ}.  
Furthermore, we assume that there are only two different values of $c_j$'s, being 0.95 and 1.05. Without underlying network, $\mu_j=1$ and $c_j<1$ yields as ruin parameter, for object $j$, $\rho = \lambda \mu_j/c_j > 1 $  so that $ \mathbb{P} (R_j(t)  \ge u \mbox{ for some } t > 0 ) = 1$ for all $u>0$ (Proposition 1.1 in \cite{asmussen2010ruin}).  Thus, if an agent only insures object $j$ and is the only agent to insure object $j$, then this agent would experience ruin.

\begin{figure}
    \centering
    \includegraphics[width = \linewidth,height=2.9in]{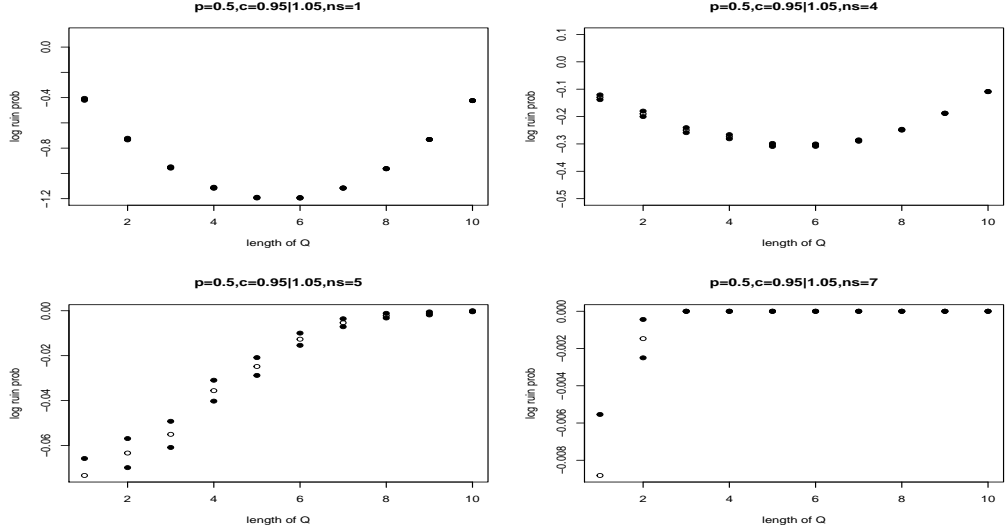}
    \caption{Logarithm of ruin probabilities in a bipartite Bernoulli network with $p=0.5$ plotted against the size of $Q$. The unfilled circles are the Monte-Carlo estimates and the filled circles are the 95\% confidence intervals.}
    \label{phase transition}
\end{figure}

Figure \ref{phase transition} shows the result of the simulations.
In the titles of the plots, $ns := \sum_{j=1}^d \mathbbm{1}\{c_j=0.95\}$ is the number of objects experiencing ruin in the univariate case. From left to right and from top to bottom, we increase $ns$. In the top two panels, the ruin probability achieves its minimum for a size of $Q$ which is neither all agents nor a singleton, indicating the possibility for selection of $Q$ which minimises the ruin probability. In the bottom two panels, the ruin probability increases with  $|Q|$. There is an abrupt 
change of shape as we increase $ns$ from four to five, in this case from U-shape to S-shape. We call this a phase transition phenomenon. 
The remainder of this section provides an explanation of this phenomenon under a regime in which the underlying network is moderately dense, so that the edge probabilities are $O(d^{-\beta})$ for some $\beta \in (0,1).$

In view of this phase transition, we want to examine \eqref{ruinexpression} as a function of $|Q|$, with the choice of $r^Q$ as in \eqref{rQ}.  If $\mathbb{P}(P^Q<1) = 0$, then $\log\Psi^Q(u) = 0,\forall |Q|$. The bottom two panels in Figure \ref{phase transition} show that for large $Q$, $\mathbb{P}(P^Q<1)$ is very close to $ 0$. To explain the top two panels, using the delta method,  it can be seen that  
if $\mathbb{P}(P^Q<1)=1$, then $\log\Psi^Q(u)$ is approximately a quadratic function in $|Q|$.
Hence
the term $\mathbb{P}(P^Q<1)$ plays an important role in  
determining the shape of $\Psi^Q(u)$. Thus, this paper focuses on $\mathbb{P}(P^Q<1)$. 

For each $j$, let $I_j = \mathbbm{1}\{Q\sim j\}$, $\rho_j = \lambda\frac{\mu_j}{c_j}$ and $\xi_j = \frac{1}{\rho_j}$. Then with \eqref{ruinparameter},
$$
    \mathbb{P}(P^Q<1) = \mathbb{P}\left(\frac{\sum_{j=1}^d I_j}{\sum_{j=1}^d \xi_jI_j} <1\right) = \mathbb{P}\left( \sum_{j=1}^d (\xi_j-1)I_j \geq 0\right),
$$
where $I_j \sim Ber(\mathbb{P}(Q\sim j))$ are identically distributed but not independent. 
This probability is
difficult to compute exactly when the number of objects $d$ is large. Instead, using Stein's method we derive a normal approximation for $\sum_{j=1}^d (\xi_j-1)I_j$, together with a bound on the approximation.


To this purpose we introduce more notation. 
Given a subset $Q$ of agents and an object $j$, denote the type configuration on $Q\cup \{j\}$ by $$C^Q_j = \{s(i)\}_{i\in Q} \cup \{t(j)\} \mbox{ and let } 
C^Q := \bigcup_{j=1}^d C^Q_j = \{s(i)\}_{i\in Q} \cup \{t(j)\}_{j=1}^d$$
denote the type configuration on $Q$ {and all objects}.
In particular, a realisation $c$ of $C^Q$ induces a realisation $c_j =\{s(i)\}_{i\in Q} \cup \{t(j)\} $ for each $C^Q_j$.
  We let $$ p(c_j):= \mathbb{P}(Q\sim j|C^Q_j =c_j) = 1-\prod_{i\in Q} (1-p_{s(i)t(j)})$$
  and
  $$  \sigma^2(c) := \sum_{j=1}^d (\xi_j-1)^2  p(c_j)(1-p(c_j) ).$$ 
  Moreover, $\Psi$ denotes the standard normal cumulative distribution function.

\begin{theorem}[Mixture of normal approximation]
\label{mixture bound} For each realisation $c$ of $C^Q$ let 
 $N(c)$ be normal distributed with mean $\sum_{j=1}^d (\xi_j-1)p(c_j)$ and variance  $\sigma^2(c) > 0 $ and let $N = N(c)$ with probability $\mathbb{P}(C^Q=c)$.
 Then 
\begin{align}\label{thebound}
    |\mathbb{P}(P^Q<1)-\mathbb{P}(N > 0)| \leq 9.4\sum_c \mathbb{P}(C^Q=c)\sum_{j=1}^d \mathbb{E}[|Z_j(c)|^3]
\end{align} 
where $Z_j(c) := \frac{1}{\sigma (c) }(\xi_j-1)(I_j-p(c_j))$.
\end{theorem}
\begin{proof}
For each  $c$,
$
\mathbb{P}(N(c) > 0) =  \Psi\left(\sum_{j=1}^d (1-\xi_j)p(c_j) / \sigma(c) \right)
$
and 
\begin{align*}
    \mathbb{P}(P^Q<1|C^Q=c) 
    &=\mathbb{P}\left(\sum_{j=1}^d(\xi_j-1)I_j\geq 0  \big\vert \, C^Q=c\right) \\
   & = -\mathbb{P}\left( W({c}) \leq \frac{1}{\sigma(c)}\sum_{j=1}^d (1-\xi_j)p(c_j)  \right) 
\end{align*}
where $W(c) := \sum_{j=1}^d Z_j(c)$.
By construction, $\mathbb{E}[W({c})]=0$ and $Var(W({c})) =1$. Moreover, given the type configuration, the edge indicators are independent.  By invoking Theorem 3.6, p.54, in \cite{chen}, it follows that 
\begin{equation}
     \sup_{z\in \mathbb{R}}|\mathbb{P}(W({c})\leq z)-\Psi(z)| \leq 9.4\sum_{j=1}^d \mathbb{E}[|Z_j({c})|^3].
 \end{equation} 
Conditioning on the type configuration, 
the claim is now immediate.
\end{proof}

Next we explore the behaviour of the bound. 

\begin{corollary}
\label{asymptotic}
Under the stochastic block model defined in Section \ref{setting}, suppose that $Q$ is fixed, $p_{kl} = \mathcal{O}(d^{-\beta})$  for all $k,l$, with $\beta \in (0,1)$, and there exist constants $0 < D_1 <D_2$ such that for all $j$, $D_1\leq |\xi_j-1| \leq D_2$. Then 
$
\lim_{d\to\infty} |\mathbb{P}(P^Q < 1)-\mathbb{P}(N > 0)| = 0.
$

Assume further that there is a constant $0 <  B < \infty$ such that  for all configurations $c$, $
{p(c_j)} \le B d^{-\beta}$, with $\beta \in (0,1)$. Moreover assume that there exist constants $m < 0 < M$ such that 
for every realisation $c$ of $C^Q$, there exist $\mu (c) \in (-\infty, m) \cup (M, \infty) $ such that $$\lim_{d\to\infty}\left|  \frac{1}{d} \sum_{j=1}^d d^{\beta} (\xi_j-1) p(c_j) - \mu (c) \right| = 0 $$ uniformly for all $c$. 
then
\begin{equation} \label{secondassertion} 
    \lim_{d\to \infty} |  \mathbb{P}(P^Q<1) - \mathbb{P}(\mu( C^Q) >0) | =0.
\end{equation}
\end{corollary}

\begin{proof}
It is easy to show that 
$$\mathbb{E}[|Z_j({c})|^3] = \frac{|\xi_j-1|^3}{\sigma^3(c)}\left(p(c_j)(1-p(c_j))^3 + (1-p(c_j)p(c_j)^3)\right).$$
If for every $k,l$, $p_{kl} = \mathcal{O}(d^{-\beta})$ with $\beta \in (0,1)$, then for fixed $Q$, 
\begin{equation}
\label{asymp_prob}
    p(c_j) = 1-\prod_{i\in Q}(1-p_{s(i)t(j)}) = \sum_{i\in Q} p_{s(i)t(j)} + h.o.t = \mathcal{O}(d^{-\beta}).
\end{equation}
As for all $j$, $D_1\leq |\xi_j-1| \leq D_2$,  the bound \eqref{thebound} in Theorem \ref{mixture bound} gives,
\begin{align*}
    9.4&\sum_c \mathbb{P}(C^Q=c)\frac{\sum_{j=1}^d |\xi_j-1|^3(p(c_j)(1-p(c_j))^3 + (1-p(c_j))p(c_j)^3)}{[\sum_{j=1}^d (\xi_j-1)^2 p(c_j)(1-p(c_j))]^{3/2}}\\
    & \leq \frac{9.4 D_2^3}{D_1^3}\sum_c \mathbb{P}(C^Q=c)\frac{\sum_{j=1}^d (p(c_j)(1-p(c_j))^3 + (1-p(c_j))p(c_j)^3)}{[\sum_{j=1}^d  p(c_j)(1-p(c_j))]^{3/2}}\\
    &= \mathcal{O}(d^{-(1-\beta)/2})
\end{align*}
which tends to zero as $d$ tend to infinity. The first assertion follows.

Next, note that
$
    \mathbb{P}(N(c)> 0)
     = \Psi\left({d^{1-\beta} \frac{1}{d}\sum_{j=1}^d d{^\beta} (\xi_j-1)p(c_j)}/ {{\sigma(c)} }\right).  
$
{Due to uniform convergence,  if $\mu(c)> M $ then there exists a $D$ such that for all $d> D$ and for all  $c$,} $\left|\frac1d \sum_{j=1}^d d^{\beta}(\xi_j-1)p(c_j) - \mu(c)\right| \leq \frac{M}{2}$, and thus 
$$
\frac1d \sum_{j=1}^d d^{\beta}(\xi_j-1)p(c_j) \geq \mu(c) - \frac{M}{2} > \frac M 2.
$$

As  $
    \Psi( t) \ge 1 -  \min \left\{ \frac12, \frac{1}{t \sqrt{2\pi} }  e^{-t^2/2} \right\} \ge  1 -  \min \left\{ \frac12, \frac{1}{t \sqrt{2\pi} } \right\} $  for any $t >0$, (see  Eq. (2.11), p.16, in \cite{chen}), if $\mu(c) > M$ then for $d > D$, 
\begin{align*}
    \mathbb{P}(N(c) >   0) 
    &   \ge 1-   \min  \left\{ \frac12, \frac{\sigma (c) }{ d^{1 - \beta} \sqrt{2\pi} \frac{1}{d}\sum_{j=1}^d (\xi_j-1)p(c_j) d^\beta  } 
    \right\}\\
    & {\geq 1- \min \left\{ \frac12, \frac{2 D_2 \sqrt{B} d^{(1-\beta)/2}}{d^{1-\beta} \sqrt{2\pi} M}\right\}} 
\end{align*}
{where the last inequality is due to $\sigma(c) \leq D_2 \sqrt{B}  (d^{(1-\beta)/2})$; }
this bound is independent of $c$. 
As  $\frac{ D_2 \sqrt{{B}} d^{(1-\beta)/2}}{ d^{1 - \beta} \sqrt{2\pi} M } \rightarrow 0 $  for $d \rightarrow \infty$,  
$  \mathbb{P}(N(c) >   0)  \rightarrow 1$ uniformly for all configurations $c$ such that 
$\frac{1}{d}\sum_{j=1}^d  d^{\beta} (\xi_j-1)p(c_j)  \rightarrow \mu(c) >  M $. Similarly we can show that 
$  \mathbb{P}(N(c) >   0)  \rightarrow 0$ uniformly for all configurations $c$ such that 
$\frac{1}{d}\sum_{j=1}^d d^{\beta} (\xi_j-1)p(c_j)  \rightarrow \mu(c) < m$.
Thus, 
\begin{align*}
\lefteqn{\left|   \mathbbm{P}(N > 0) - \mathbbm{P}( \mu (C^Q) > 0) \right| } \\
    & =   \sum_c \mathbbm{P}(C^Q = c) \left\{ \mathbbm{I}( \mu (c) > 0) ( 1 -  \mathbbm{P}(N (c) > 0) )  
    + \mathbbm{I}( \mu (c) < 0)  \mathbbm{P}(N (c) > 0)  \right\} 
\end{align*}
which tends to 0 as $d \rightarrow \infty$.  The triangle inequality combined with the first assertion  gives \eqref{secondassertion}.

\end{proof}



\begin{example}[A bipartite Bernoulli network] If  there is only one type for both the agents and objects in the stochastic block model, then there is only one type configuration, ${c^*}$, and the model reduces to a bipartite Bernoulli network with  edge probability $p=p_{11}$. Assume that 
$\lim_{d\to\infty} p d^{\beta} = A > 0$. Then from \eqref{asymp_prob}, $\lim_{d\to\infty} p({c^{*}}) d^{\beta} = A|Q|$.
If there exists $E \in \mathbbm{R} \setminus \{0\}$ 
such that 
$\lim_{d\to \infty} \frac1d \sum_{j=1}^d (\xi_j-1) = E$, then the assumptions of Corollary \ref{asymptotic} are satisfied, and, as $d \rightarrow \infty,$ 

\begin{equation*}
    \mathbb{P}(P^Q<1) \to
	\begin{cases}
	1 & \text{if } E >0;\\
	0 & \text{if } E <0.
	\end{cases}
\end{equation*}
Table \ref{table} illustrates the bound and the approximation when  
the number of objects is $d=100,000$, $|Q|=100$, $\beta = \frac{1}{2}$, $p = d^{-\beta}$, $\lambda =1$, $\mu_j = 1$ for all $j$ and $c_j = 0.95$ or $1.05$. In the columns of Table \ref{table}, $\textbf{ns} = \sum_{j=1}^d \mathbbm{1}\{c_j=0.95\}$ is the number of objects with $c_j=0.95$ (which in the classical Cram\'{e}r-Lundberg model would lead to ruin), \textbf{bound} is calculated from Theorem \ref{mixture bound}, \textbf{approximation} is the value of $\mathbb{P}(N>  0)$, \textbf{estimate} is a Monte Carlo estimate of $\mathbb{P}(P^Q<1)$ and \textbf{abs difference} is the absolute difference between approximation and estimate.
For
$ns<50,000$, we have
$\sum_{j=1}^d (\xi_j-1) /d > 0$ and  for $ns>50,000$, we have $ \sum_{j=1}^d (\xi_j-1)/d < 0$.
In both regimes 
the absolute difference is always much smaller than the bound  obtained from Theorem \ref{mixture bound}, but this bound is of similar magnitude as the error bars from the Monte Carlo simulations, with the latter not giving a theoretical guarantee.
\begin{center}
\begin{table}
\caption{Behaviour of the bounds in a bipartite Bernoulli graph }\label{table}
    \begin{tabular}{|c|c|c|c|c|} 
    \hline
    $ns$ & bound  & approximation & estimate ($\pm 2\times $ s.d.) & abs difference\\
    \hline
    49000 & 0.040 & 1.000 & 1.000 & 0 \\
    49500 & 0.040 & 0.973 & $0.979 (\pm 0.009)$ & 0.006 \\
    49900 & 0.040 & 0.650 & $0.659 (\pm 0.030)$ & 0.008 \\
    50000 & 0.040 & 0.500 & $0.504 (\pm 0.032)$ & 0.004 \\
    50100 & 0.040 & 0.350 & $0.354 (\pm 0.030)$ & 0.004 \\
    50500 & 0.040 & 0.027 & $0.018 (\pm 0.008)$ & 0.009 \\
    51000 & 0.040 & 0.000 & 0.000 & 0\\
    \hline
    \end{tabular}
    \end{table} 
\end{center}
Moreover, Corollary \ref{asymptotic} can be used to explain the changing shape in  Figure \ref{phase transition}. In the setting for this figure, for  $ns < 5$ we have $\sum_{j=1}^d (\xi_j-1)/d > 0$.
Approximately, then, 
$\mathbb{P}(N >  0) > 0.5$, and by Theorem \ref{mixture bound}, $\mathbb{P}(P^Q < 1) > 0.5$.
Thus,  in \eqref{ruinexpression}, the first term, including the expectation, dominates, in a limiting regime under which the bound in Theorem \ref{mixture bound} is small. Using the delta method, the logarithm of this expectation is approximately a quadratic function in ${|}Q{|}$ and thus  
results in the "U-shape". In contrast, when $ns > 5$, approximately,  $\mathbb{P}(N > 0) < 0.5$, and the expression is dominated by $\mathbb{P}(P^Q \ge 1)$, which for large ${|}Q{|}$ results, in a straight line, hence the "S-shape".

\end{example} 

\begin{remark}When $\beta \ge1$ the bound in {Theorem \ref{mixture bound}} does not converge to zero as $d \rightarrow \infty$. For sparse networks a Poisson approximation may be more appropriate, approximating 
the collection of dependent Bernoulli edge indicators   
by a collection of independent Poisson variables
with matching means.
As conditional on the type configuration, the indicators are independent,  Theorem 1.A from \cite{barbour1992poisson}   gives a bound in total variation distance for this approximation; 
if $p_{kl} = \mathcal{O}(d^{-1})$ for all $k,l$, 
then this bound goes to zero as $d$ tends to infinity. Studying the existence of a corresponding phase transition in this regime remains an open problem. 


\end{remark} 

\medskip
{\bf Acknowledgement.} The author would like to thank Gesine Reinert for her guidance and contribution by taking part in many helpful discussions about the results, and for her comments on the manuscript.

\bibliographystyle{elsarticle-harv} 
\bibliography{ref}






\end{document}